\documentclass[11pt]{article}

\usepackage{amsmath}
\usepackage{amsfonts}
\usepackage{amssymb}
\usepackage{amsthm}
\usepackage{breqn}
\usepackage{setspace}
\usepackage{fullpage}
\usepackage{enumitem}
\usepackage{bbold} 
\usepackage{comment}
\usepackage{hyperref}

\newtheorem{theorem}{Theorem} 
 
\newtheorem{definition}{Definition}

\newtheorem{remark}{Remark}

\newcommand{\A}{\mathcal{A}}
\newcommand{\B}{\mathcal{B}}
\newcommand{\C}{\mathcal{C}}

\newcommand{\I}{\mathcal{I}}
\newcommand{\p}{\mathcal{P}}

\newcommand{\F}{\mathcal{F}}

\newcommand{\LL}{\mathcal{L}}

\newcommand{\abs}[1]{\left\lvert{#1}\right\rvert}
\newcommand{\floor}[1]{\left\lfloor{#1}\right\rfloor}

\DeclareMathOperator{\La}{La}

\title{On an extremal problem involving a pair of forbidden posets}
 \author
{
Casey Tompkins\thanks{Alfr\'ed R\'enyi Institute of Mathematics, Hungarian Academy of Sciences.  e-mail: ctompkins496@gmail.com} \thanks{This research was supported by the National Research, Development and Innovation Office --- NKFIH under the grant K116769.}
\and
Yao Wang\thanks{Department of Mathematics, Texas A\&M University. e-mail: wangyao@math.tamu.edu }
}

\begin{document}
\maketitle
\begin{abstract}
Resolving a conjecture of Methuku and the first author we determine the size of the largest family of subsets of an $n$-element set avoiding both $Y_k$ and $Y_k'$ as induced subposets. The result follows as a consequence of the analogous result on a cyclical grid poset.  
\end{abstract}
\section{Introduction}
Let $[n]=\{1,2,\dots,n\}$ and for any set $X$, let $2^X$ denote the power set of $X$ and $\binom{X}{i}$ denote the set of all $i$-element subsets of $X$.  The sets $\binom{[n]}{i}$ are referred to as levels in the Boolean lattice $2^{[n]}$. We use the notation $\Sigma(n,k)$ for the sum of the sizes of the $k$ largest levels in the Boolean lattice $2^{[n]}$.

Inspired by classical results of Sperner \cite{sperner1928satz} and Erd\H{o}s \cite{erdos1945lemma}, Katona and Tarj\'an \cite{katona1983extremal} launched the systematic study of extremal forbidden poset problems in the Boolean lattice.  Over the last decade, this field has undergone substantial development with the extremal function of a variety of posets being determined exactly or asymptotically.   To state our results we will need to recall some definitions.  

For partially ordered sets (posets) $P$ and $Q$ we say that $P$ is a \emph{subposet} of $Q$ if there exists an injection $\phi:P \to Q$ such that $x<y$ in $P$ implies $\phi(x) < \phi(y)$ in $Q$.  We say that $P$ is an \emph{induced subposet} of $Q$ if there exists an injection $\phi$ such that $x<y$ in $P$ if and only if $\phi(x)<\phi(y)$ in $Q$. We will regard collections of sets $\A \subset 2^{[n]}$ as posets with respect to the containment relation.  We say that $\A$ is $P$-free if $P$ is not a subposet of $\A$, and $\A$ is induced $P$-free if $P$ is not an induced subposet of $\A$. Then we may define the following extremal functions:
\begin{definition}
For a collection of posets $\p$, we define
\begin{displaymath}
\La(n,\p) = \max \{\abs{\A}:\A \subset 2^{[n]} \mbox{  and $\A$ is $P$-free for all $P\in\p$}\}
\end{displaymath}
and
\begin{displaymath}
\La^*(n,\p) = \max \{\abs{\A}:\A \subset 2^{[n]} \mbox{  and $\A$ is induced $P$-free for all $P\in\p$}\}.
\end{displaymath}
When $\p = \{P\}$, we write simply $\La(n,P)$ (and similarly $\La^*(n,P)$) in place of $\La(n,\p)$ ($\La^*(n,\p)$).
\end{definition}
Already in the first paper on the topic, Katona and Tarj\'an proved the following extremal result for a pair of forbidden posets.

\begin{theorem}[Katona-Tarj\'an \cite{katona1983extremal}]
\label{KT}
Let $V$ and $\Lambda$ be the posets on three elements $x,y,z$ defined by the relations $x<y,z$ and $x,y<z$, respectively.  Then,
\begin{displaymath}
\La(n,\{V,\Lambda\}) = \La^*(n,\{V,\Lambda\}) = 2 \binom{n-1}{\floor{\frac{n-1}{2}}}.
\end{displaymath}
\end{theorem}
\noindent An extremal construction for Theorem \ref{KT} is given by the family 
\begin{displaymath}
\binom{[n-1]}{\floor{\frac{n-1}{2}}} \cup \left\{F\cup\{n\}: F \in \binom{[n-1]}{\floor{\frac{n-1}{2}}} \right\}.
\end{displaymath}
Note that if $n$ is even, then taking the middle level $\binom{[n]}{n/2}$ is also extremal.

Let $Y_k$ denote the poset with elements $x_1,x_2,\dots,x_k,y,z$ such that $x_1 < x_2 <\dots<x_k<y,z$, and let $Y_k'$ denote the same poset but with all relations reversed.  Let $B$, the so-called butterfly poset, be defined on $w,x,y,z$ with relations $w,x < y,z$.  De Bonis, Katona and Swanepoel \cite{de2005largest} proved that $\La(n,B) = \Sigma(n,2)$, with the middle two levels of $2^{[n]}$ giving an extremal construction.  In fact, their proof directly implies the stronger result $\La(n,\{Y_2,Y_2'\}) = \Sigma(n,2)$. (This result is stronger since $B$ is a subposet of $Y_2$ and $Y_2'$.) 

In \cite{methuku2014exact} the first author and Methuku proved two natural extensions of this result.  Namely, $\La(n,\{Y_k,Y_k'\}) = \Sigma(n,k)$ was proved, thereby generalizing the result of De Bonis, Katona and Swanepoel to all $k$, as well as $\La^*(n,\{Y_2,Y_2'\}) = \Sigma(n,2)$, yielding a stronger result in the case $k=2$.  Constructions, yielding matching lower bounds are given by the union of $k$ largest levels from $2^{[n]}$.  In the same paper, it was conjectured that the common generalization of these two results should hold: $\La^*(n,\{Y_k,Y_k'\}) = \Sigma(n,k)$. Our main result is a proof of this conjecture.  
\begin{theorem}
\label{bn}
For all $k \ge 2$ and $n \ge k+1$, we have $\La^*(n,\{Y_k,Y_k'\}) = \Sigma(n,k)$.  
\end{theorem}
\begin{remark}
During the preparation of the manuscript we learned Theorem \ref{bn} has been obtained independently by Martin, Methuku, Uzzell and Walker \cite{martin2017simple}.  The methods of proof are different: They use a discharging argument whereas we use cyclic permutations.  
\end{remark}
Theorem \ref{bn} demonstrates that the extremal function $\La^*(n,\{Y_k,Y_k'\})$ behaves differently in the $k \ge 2$ case than in the $k=1$ case (Katona and Tarj\'an's result on $\{V,\Lambda\}$). Theorem \ref{bn} (as well as the results of \cite{methuku2014exact}) are also a generalization of the theorem of Erd\H{o}s \cite{erdos1945lemma} which states that a family of sets with no $(k+1)$-chain ($A_1\subset A_2 \subset \dots \subset A_{k+1}$) can have size at most $\Sigma(n,k)$.

In \cite{methuku2014exact} the value of $\La(n,\{Y_k,Y_k'\})$ was determined using Katona's \cite{katona1972simple} method of cyclic permutations. In this paper we will use the same method to determine $\La^*(n,\{Y_k,Y_k'\})$.  We begin by recalling the basic notions involved in this approach.  

Let $\I(n)$ denote the collection of sets formed by arranging the numbers $1,2,\dots,n$ in order around a circle and taking those sets (excluding $\varnothing$ and $[n]$) which form intervals along this cyclic arrangement (e.g., $\{2,3,4\}$ or $\{n-1,n,1,2\}$). Let $\pi \in S_n$ be a permutation and set $\I(n)^\pi = \{\{\pi(a_1),\pi(a_2),\dots,\pi(a_r)\}:\{a_1,a_2,\dots,a_r\} \in \I(n)\}$.  A now standard double-counting argument of Katona \cite{katona1972simple} shows that if for any family $\A \subset 2^{[n]}$ such that $\varnothing,[n] \not\in \A$ we have $\A \cap I(n)^\pi \le k$ for all $\pi$, then $\sum_{A \in \A} \frac{1}{\binom{n}{\abs{A}}} \le k$.  We prove that this is indeed the case for induced $Y_k,Y_k'$-free families.

\begin{theorem}
\label{cycle}
For all $n \ge k+1$ where $k \ge 2$ and all $\pi \in S_n$, if $\A \subset 2^{[n]}$ is an induced $Y_k,Y_k'$-free family of sets, then $\abs{\I(n)^\pi \cap \A} \le kn$. 
\end{theorem}

Theorem \ref{cycle} will be proved in Section \ref{pf}.  As a consequence of Theorem \ref{cycle} we have the following LYM-type inequality.

\begin{theorem}
\label{LYM}
Let $n \ge k+1$ where $k \ge 2$ and let $\A \subset 2^{[n]}$ be an induced $Y_k,Y_k'$-free collection of sets such that $\varnothing,[n] \not\in \A$, then
\begin{equation}
\label{LYMineq}
\sum_{A \in \A} \frac{1}{\binom{n}{\abs{A}}} \le k.
\end{equation}
\end{theorem}

From the LYM-type inequality Theorem \ref{LYM} we can easily deduce our main result.

\begin{proof}[Proof of Theorem \ref{bn}]
If $\B$ is an family of sets satisfying the inequality \eqref{LYMineq}, then it is easy to see that $\abs{\B} \le \Sigma(n,k)$.  Moreover, if $\B$ contains a set $B$ whose size is not the size of a set in the $k$ largest levels of $2^{[n]}$, then $\abs{\B} \le \Sigma(n,k)-1$.

Let $\A$ be an induced $Y_k,Y_k'$-free family (assume $k\ge 3$; the $k=2$ case was settled in \cite{methuku2014exact}). We will assume inductively that the result holds for smaller values of $k$.  

We are done by Theorem~\ref{LYM} if neither $\varnothing$ nor $[n]$ is in $\A$. Suppose, now, that both $\varnothing$ and $[n]$ are in $\A$.   The family $\B = \A \setminus \{\varnothing,[n]\}$ is an induced $Y_{k-1},Y_{k-1}'$-free family and so $\abs{\B} \le \Sigma(n,k-1)$.  It follows that $\abs{\A} \le 2 + \Sigma(n,k-1) < \Sigma(n,k)$.  

Finally, suppose $\varnothing \in \A$ and $[n] \not\in \A$ (the remaining case follows by taking complements).  Let $\B = \A \setminus \{\varnothing\}$. If $\abs{\B} \le \Sigma(n,k)-1$ we are done. Suppose $\abs{\B} = \Sigma(n,k)$, then $\B$ is a subset of the $k$ (or $k+1$) largest levels in $2^{[n]}$.  If $n \neq k$ modulo 2, then $\B$ is uniquely determined, and it is easy to find an induced $Y_k$ in $\A$.  Indeed, let $\LL_1,\LL_2,\dots,\LL_k$ be the $k$ largest levels of $2^{[n]}$, ordered by the cardinality of the sets belonging to each level. Then $\B = \cup_{i=1}^k \LL_i$, and choosing an arbitrary chain $L_1 \subset L_2 \subset \dots \subset L_{k-1}$ where $L_i \in \LL_i$ along with two sets $A \in \LL_k$ and $B \in \LL_k$ with $L_{k-1} = A \cap B$ yields an induced $Y_k$ consisting of $\varnothing,L_1,L_2,\dots,L_{k-1},A$ and $B$.

Suppose $n =  k$ modulo 2.  Let $\LL_1,\LL_2,\dots,\LL_{k+1}$ be the $k+1$ levels in $2^{[n]}$ of largest size, ordered by the cardinality of the sets belonging to each level.  Then $\A = \cup_{i=2}^{k} \LL_i \cup \C$ where $\C \subset \LL_1 \cup \LL_{k+1}$ and $\abs{\C} = \abs{\LL_1} = \abs{\LL_{k+1}}$.  If $\C \cap \LL_1 = \varnothing$, then we have $\C = \LL_2$ and and we find an induced $Y_k$ as in the previous case.  If $\C \cap \LL_1 \ne \varnothing$, then let $C \in \C \cap \LL_1$.  Take a $(k-1)$-chain $C \subset L_2 \subset L_3 \subset \dots \subset L_{k-1}$ where $L_i \in \LL_i$, and take $A,B\in \LL_{k}$ such that $L_{k-1} = A \cap B$.  Then $\varnothing,C,L_2,L_3,\dots,L_{k-1},A$ and $B$ yields an induced $Y_k$.
\end{proof}

\section{Proof of Theorem \ref{cycle}}
\label{pf}
Without loss of generality, we may assume $\pi$ is the identity and bound the size of $\A \cap \I(n)$. That is we must bound the number of intervals we can take along the cyclic arrangement $1,2,\dots,n,1$.   Following the reasoning of \cite{methuku2014exact}, we partition $\I(n)$ into chains of intervals in two different ways.  We take the set of all $n$ chains of the form $\{\{i\},\{i,i+1\},\dots,\{i,i+1,\dots,i+n-2\}\}$ (additions taken modulo $n$), $i \in [n]$, which we call \emph{ascending chains}.  We also take the set of all chains of the form $\{\{i\},\{i,i-1\},\dots,\{i,i-1,\dots,i-n+2\}\}$, $i \in [n]$, which we call \emph{descending chains}.  The set of all $n$ ascending chains together with the set of all $n$ descending chains forms a set of $2n$ chains.  Both the set of ascending chains and the set of descending chains form partitions of $\I(n)$.  Each set $X \in \I(n)$ appears exactly once in each type of chain.  In order to prove $\abs{\A \cap \I(n)} \le kn$, it is sufficient to prove that the average size of an intersection of $\A$ with one of the $2n$ ascending or descending chains is at most $k$ (see \cite{methuku2014exact} for an explicit calculation).

The following observation is simple, but critical for finding induced $Y_k$'s and $Y_k'$'s. If $C$ is an ascending chain crossing a descending chain $C'$ at a set $X$, then every pair of sets $Y\in C$ and $Z \in C'$ with $\abs{Y},\abs{Z} >\abs{X}$ are incomparable (and similarly for such pairs smaller than $\abs{X}$). Indeed, by the definition of the chains if $X = \{x,x+1,\dots,x+t\}$, then $x-1 \in Z$ but $x-1 \not\in Y$, whereas $x+t+1 \in Y$ but $x+t+1 \not\in Z$ (a similar argument holds for the case $\abs{Y},\abs{Z}<\abs{X}$).  

We will partition the $2n$ chains into classes which we refer to as groupings (usually consisting of both ascending and descending chains).  To this end, we introduce some notation which we will use to determine which chain goes in which grouping.   Let $C$ be an ascending chain containing at least one set from $\F$.  The predecessor of $C$ is the unique descending chain $C'$ containing the set $\min (C \cap \F)$.  Let $h(C)$ be the number of sets in $C'\cap \F$ which are smaller than $\min (C \cap \F)$.   Similarly, for a descending chain $C$, the predecessor $C'$ of $C$ is the unique ascending chain containing the set $\max(C \cap \F)$, and $h(C)$ is the number of sets in $C' \cap \F$ which are larger than $\max(C \cap \F)$.  Given an ascending or descending chain $C$ containing a set $X$, we call the other chain $C'$ which contains $X$ the chain crossing $C$ at $X$.  

For any chain $C$ we define the length of $C$, $\ell(C)$, to be $\abs{C \cap \F}$.  For an ascending chain $C$ of length $s$, we denote the elements of $C \cap \F$ by $C^1,C^2,\dots,C^s$ where  $C^1\subset C^2 \subset \dots \subset C^s$. For a descending chain $C$ of length $s$, we denote the elements of $C \cap \F$ by $C^1,C^2,\dots,C^s$ where  $C^1\supset C^2 \supset \dots \supset C^s$. For a collection $E$ of chains, the weight of $E$ is $w(E) = \sum_{C\in E} (\ell(C)-k)$.  Our goal will be to construct a series of groupings whose combined weight is at most $0$. 

We will begin by considering those chains $C$ with $\ell(C) \ge k+1$ such that $h(C) \ge k-1$.  Let $\ell(C) = k + r$ where  $r\ge 1$.  The chains crossing $C$ at $C^2,C^3,\dots,C^{r+1}$ have only one set from $\F$.  Forming a group $E$ consisting of $C$ along with these $r$ chains, we have
\begin{displaymath}
w(E) = r + r(1-k) \le 0.  
\end{displaymath}
Call such a grouping $E$ a \emph{type 1} grouping.  Form a type 1 grouping for every such chain $C$.

Next, we consider those chains $C$ with $\ell(C) \ge 2k-1-h(C)$ where $h(C) \le k-2$. Let $\ell(C) = 2k-1-h(C) + r$ where $r\ge 0$.  The chains crossing $C$ at $C^{k-h(C)},C^{k-h(C)+1},\dots,C^{k-h(C)+r}$ have only one set from $\F$. Forming a grouping $E$ consisting of $C$ along with these $r+1$ chains yields a weight
\begin{align*}
w(E) &= k-1 - h(C) + r + (r+1)(1-k) \\
&= -h(C) + r(2-k) \\
&\le -h(C).
\end{align*}
Call such a grouping $E$ a \emph{type 2} grouping.  Form a type 2 grouping for every such chain $C$. 

Finally, we consider those chains with $k+1 \le \ell(C) \le 2k-2-h(C)$ where $h(C) \le k-2$. We will define our grouping (which we call a \emph{type 3} grouping) in a step-by-step way, adding chains one at a time.  We introduce two parameters $m$ and $t$ which we will define inductively as we create the grouping.  Let $C=C_1$ be an ascending chain (the descending case is completely analogous) with $k+1 \le \ell(C) \le 2k-2-h(C)$, and set $m(C_1) = k-h(C_1)$.  Then, $2 \le m(C_1) \le k$ and the chain $C_2$ crossing $C_1$ at $C_1^{m(C_1)}$ contains no element of $\F$ larger than $C_1^{m(C_1)}$ (for otherwise we would have an induced $Y_k$ since we would have $h(C_1)$ sets from the predecessor of $C_1$ and $m(C_1)$ sets from $C_1$ all comparable to each other, and contained in at least one more set from $C_1$ as well as contained in an incomparable set from $C_2$ ).  We define $t(C_1)$, the size of the ``tail'' of $C_1$, to be the number of sets in $C_1 \cap \F$ larger than $C_1^{m(C_1)}$.   Then, 
\begin{equation}
t(C_1) = \ell(C_1) - m(C_1) = \ell(C_1)-k+h(C_1).
\end{equation}
Note that $1 \le t(C_1) \le k-2$.

If $C_2$ is already in some earlier grouping then  we stop and define our current grouping $E$ by $E = \{C_1\}$. In this case $w(E) = t(C_1) - h(C_1)$.

If $C_2$ is not in some earlier grouping we continue.  Now, $t(C_1)$ will play a completely analogous role for $C_2$ as $h(C_1)$ played for $C_1$.  Define $m(C_2)$ by $m(C_2) = k-t(C_1)$. If $m(C_2) \ge \ell(C_2)$ we stop and let $E = \{C_1,C_2\}$.  We have
\begin{align*}
w(E) = (\ell(C_1) - k) + (\ell(C_2)-k) &\le (t(C_1) - h(C_1)) + (m(C_2)-k) \\
&= -h(C_1).
\end{align*}
If $m(C_2) \le \ell(C_2)-1$ we consider the chain $C_3$ crossing $C_2$ at $C_2^{m(C_2)}$.  Let $t(C_2) = \ell(C_2) -m(C_2)$.  If $C_3$ has already been considered we stop and let $E=\{C_1,C_2\}$ (we will calculate the weight in general later) otherwise we continue with $C_3$.  Observe that $C_3$ contains no element of $\F$ smaller than $C_2^{m(C_2)}$ for otherwise we would have an induced $Y_k'$.  Observe that $2 \le m(C_2) \le k-1$.  Since $C_2$ has at most $2k-2-h(C_2)$ elements (we have already dealt with the cases of longer chains), and $h(C_2) = t(C_1)$, it follows that $t(C_2) = \ell(C_2) - m(C_2) \le k-2$.  

Continuing in this way, suppose we have already encountered $C_1,C_2,\dots,C_s$ and are now considering $C_{s+1}$ crossing $C_s$ at $C_s^{m(C_s)}$.   In each preceding step we have $m(C_{i+1}) = k - t(C_i)$ and $t(C_{i+1}) = \ell(C_{i+1})-m(C_{i+1}) = \ell(C_{i+1})-k + t(C_i)$. We also have $2 \le m(C_i) \le k-1$ and $1 \le t(C_i) \le k-2$.  Consider the chain $C_{s+1}$.  If it has already been considered by an earlier grouping (or is equal to $C_1$), then form the grouping $E = \{C_1,C_2,\dots,C_s\}$ with weight
\begin{align*}
w(E) &= \sum_{i=1}^s (\ell(C_i) - k) \\
&= t(C_1) - h(C_1) + \sum_{i=2}^s (t(C_i)-t(C_{i-1}))\\
&= t(C_s) - h(C_1).
\end{align*}

If $C_{s+1}$ has $m(C_{s+1}) \ge \ell(C_{s+1})$, then we let $E = \{C_1,C_2,\dots,C_{s+1}\}$ and we have
\begin{align*}
w(E) &= \sum_{i=1}^{s+1} (\ell(C_i) - k) \\
&= t(C_1) - h(C_1) + \sum_{i=2}^{s} (t(C_i)-t(C_{i-1})) + (\ell(C_{s+1}) - k) \\
&\le t(C_1) - h(C_1) + \sum_{i=2}^{s} (t(C_i)-t(C_{i-1})) + (m(C_{s+1})-k)\\
&= -h(C_1),
\end{align*}
since $m(C_{s+1}) = k - t(C_s)$.

If $C_{s+1}$ has $m(C_{s+1}) \le \ell(C_{s+1})-1$, then we continue with $C_{s+1}$.  Now, $m(C_{s+1}) = k - t(C_s)$ and $1 \le t(C_s) \le k-2$ so it follows that $2 \le m(C_{s+1}) \le k-1$. Also, since $\ell(C_{s+1}) \le 2k - 2 - t(C_s)$ (for otherwise $C_{s+1}$ would have been considered earlier), we have 
\begin{displaymath}
t(C_{s+1}) = \ell(C_{s+1}) - m(C_{s+1}) \le 2k-2-t(C_s) - (k - t(C_s)) = k-2,
\end{displaymath}
so $1 \le t(C_{s+1}) \le k-2$.

We continue to create type 3 groupings until all remaining chains of length greater than $k$ have been used.  Observe that the only groupings $E=\{C_1,C_2,\dots,C_s\}$ which may have positive weight are those type 3 groupings which terminated because we encountered a previously used chain.  In this case the grouping had a weight of $t(C_s)-h(C_1)$.  Let us form a directed graph on the groupings.  If a grouping $E$ is terminated because we would use a chain from $E'$, put a directed edge from $E$ to $E'$.  Observe that a grouping $E$ can can have at most one incoming edge and one outgoing edge.  Thus, the graph consists of directed paths and cycles (including 1-vertex loops).  The sum of the weights of the groupings in a cycle is clearly 0.  If a grouping has an incoming edge but not an outgoing one, then it must have weight at most $-h(C_1)$ (observe that encountering a type~1 grouping is not possible since the first chain of such a grouping intersects its predecessor at its largest or smallest element).  Thus, in the path case the total weight is at most $-h(C_1)$ where $C_1$ is the first chain in the first grouping (the one with no incoming edge).  

Since any chain which has not been considered in a grouping has weight at most 0, and the sum of the weights of the groupings is at most $0$, we have the desired result.

\bibliography{YinvertedY}
\end{document}